\documentclass[12pt, leqno]{article}
\usepackage{amsmath,amsthm,amsfonts,amssymb,tabularx}
\usepackage[latin1]{inputenc}
\usepackage{graphicx}
\usepackage{color}
\usepackage{nameref}
\usepackage{pgf,tikz}
\usepackage{pdfsync}
\usetikzlibrary[patterns]

\usetikzlibrary{arrows}

\usepackage[inner=2.5cm,outer=4.5cm,bottom=5cm,top=3cm]{geometry}

\newtheorem{teor}{Theorem}[section]

\newtheorem{note}{Note}

\newtheorem{cor}[teor]{Corollary}

\newcommand\F{\mathcal{F}}

\newcommand\laps{(-\Delta)^{s}}

\newcommand\m{\widetilde{m}}
\newcommand\mh{\widehat{m}}
\newcommand\mo{\mathring{m} }
\newcommand\no{\mathring{n} }
\newcommand\s{\widetilde{s}}
\newcommand\so{\mathring{s} }
\newcommand\sh{\widehat{s}}
\newcommand\FT{\mathcal{F}}
\newcommand\betau{\overline{\beta}}
\newcommand\alphau{\overline{\alpha}}

\numberwithin{equation}{section}

\newcommand{\RN}{\mathbb{R}^N}
\newcommand{\ren}{\mathbb{R}^N}

\begin{document}

\title{\bf Transformations of Self-Similar Solutions  \\ for porous medium equations of  fractional type}

\author{\bf Diana Stan, F\'elix del Teso  and Juan Luis V\'azquez}

\maketitle

\begin{abstract}
We consider four different models of nonlinear diffusion equations involving fractional Laplacians
and study the existence and properties of classes of self-similar solutions. Such solutions are an important
tool in developing the general theory. We introduce a number of transformations that allow us to map complete classes of solutions of one equation into those of another one, thus providing us with a number of new solutions, as well as interesting connections. Special attention is paid to the property of finite propagation.
\end{abstract}



\section{Introduction}
We will consider the following models of  evolution equations of diffusive type involving at the same type fractional Laplacian operators and power nonlinearities:
\[
u_t+\laps u^m=0, \quad \tag{M1}\label{FPME}
\]
\[
v_t=\nabla \cdot(v^{\m-1}\nabla(-\Delta)^{-\s}v ), \quad \tag{M2}\label{Model2}
\]
\[
w_t=\nabla \cdot(w\nabla(-\Delta)^{-\sh}w^{\mh-1} ),\quad \tag{M3}\label{Model3}
\]
as well as the more general version
\[
z_t=\nabla(z^{\mo-1}\nabla (-\Delta)^{-\so} z^{\no-1}). \quad \tag{MG}\label{ModelGen}
\]
 For brevity we will call them \eqref{FPME}, \eqref{Model2}, \eqref{Model3} and \eqref{ModelGen}, where M stands for Model and G for General.  Here, $(-\Delta)^{s}$ is the fractional Laplacian operator, $0<s<1$, with Fourier symbol $|\xi|^{2s}$, and $(-\Delta)^{-s}$ its inverse, cf. \cite{landkof, Stein1970}. Due to the power nonlinearities we refer to them as fractional diffusion equations of porous medium type. See details about current research  concerning Model 1 in \cite{PQRV1,PQRV2,BV2012}, cf. \cite{StanTesoVazquezCRAS,StanTesoVazquez} for Model 2, and \cite{Biler1,Biler2} for Model 3, as well as the survey papers \cite{VazSurveyFractional1,VazSurveyFractional2}. When $\widetilde{m}=\hat{m}=2$, then \eqref{Model2} and \eqref{Model3} coincide and the problem becomes $v_t=\nabla \cdot(v\nabla(-\Delta)^{-\s}v )$. Existence, finite propagation and self-similarity for this problem were recently studied in \cite{CaffVaz, CaffVaz2, CaffSorVaz}.

The behavior of the solutions of the three basic models  may be very different depending both on the equation and on the parameters $m, \m$ and $\mh$. An  efficient way of studying such differences is via the existence and properties of special solutions having particular symmetries since such solutions are either explicit or semi-explicit, or at least can be analyzed in great detail; the interest is also due to the fact that they are important in describing the properties of much wider classes of solutions.  This applies in particular to the class of self-similar solutions, namely, solutions of the types
$$
u(x,t)=t^{-\alpha}\phi(xt^{-\beta})
$$
 (so-called type I), or
 $$
 u(x,t)=(T-t)^{\alpha}\phi(x\,(T-t)^{-\beta})
 $$
 (type II). The importance of self-similar solutions in the areas of PDEs and Applied Mathematics is attested in a wide literature, cf. Barenblatt's monograph \cite{Barbk96} or \cite{JLVSmoothing}.

 The present paper is concerned with the existence, properties and correspondences of self-similar solutions of the four fractional diffusion models presented above. Important progress has been done recently in the study of self-similar solutions of these models. The  questions of existence and properties are not an easy task since in principle the solutions are not explicit.  Important questions in the qualitative analysis are the following: (i) the equation satisfied by the profile function, (ii) the decay of the profile, (iii) whether or not there is an explicit expression for $\phi$, (iv) whether the profile is compactly supported or not (finite versus infinite propagation); (v) a crucial question is the relation of these solutions to the general theory, in particular whether the large-time behaviour of a general solution of the equation is given by a self-similar solution. These are difficult questions and there are only partial answers in the literature for some of the models.

Here we  investigate the existence of transformations that enable to pass from self-similar solutions of one of the equations into  self-similar solutions of another equation, thus showing some deep connection between the models, and transferring results from one model to another one. Several coincidences had been observed in the recent literature in particular cases, see Biler et al. \cite{Biler1}, Huang \cite{Huang2013}, V\'azquez \cite{VazMesa}. We will show below that there are general transformations that apply to important classes of self-similar solutions of our models, putting whole ranges of self-similar solutions in one-to-one correspondence. We will devote special attention to the correspondence between models \eqref{FPME} and \eqref{Model2}.

Transformations between whole classes of solutions of different equations can be very useful but they are not frequent in the literature. There are however some well-known examples. Let us mention some of them in the area of nonlinear diffusion:  (i) the Hopf-Cole transformation that maps solutions of the heat equation into solutions of Burgers equation  \cite{Cole51, Hopf50}; (ii) the Lie-B\"acklund transformation that maps solutions of the fast diffusion equation into solutions of the porous medium equation in 1D, \cite{BK80}; (iii) differentiation in space maps solutions of the $p$-Laplacian equation in 1D into solutions of the porous medium equation; (iv) the relationship between these equations has been extended to the whole class of radial solutions in several dimensions by Iagar, Sánchez and Vázquez in \cite{ISV2008}.

In Section \ref{SectionVSSm2} we find explicit very singular solutions for model \eqref{Model2} of two types, in a separate variables form. The first type are solutions which are positive for all times for some values of $\m$ in the supercritical range $\m>(N-2+2\s)/N$, while the second type are solutions that extinguish in finite time for all $\m$ in the subcritical range $0<\m <(N-2+2\s)/N$.  Both types of solutions have a  form algebraically similar to the ones found by Vázquez in \cite{VazquezBarenblattFractPME} and by Vázquez and Volzone in \cite{VazquezVolzone2013Optimal} for model M1.

\normalcolor


\section{Preliminaries on Model 1}\label{SectPrelim}

This problem is probably the best known of the list. The equation is called the \emph{Fractional Porous Equation}, FPME, since it can be considered as the fractional version of the standard Porous Medium Equation $u_t=\Delta u^m$. On the other hand, for $m=1$ and $0<s<1$ we get the linear fractional heat equation, which has been also well studied.

 The existence, uniqueness and continous dependence of solutions of the Cauchy problem \eqref{FPME} for all $m>0$ and $0<s<1$ have been proved by De Pablo, Quirós, Rodríguez and Vázquez in \cite{PQRV1,PQRV2}. The main result of interest here is the property of infinite speed of propagation: Assume $s\in (0,1)$, $m>m_c=(N-2s)_+/N$. Then for non-negative initial data $u_0 \ge 0$, $\int_{\ren} u_0 (x)\,dx<\infty$, there exists a unique solution $u(x,t)$ of problem \eqref{FPME} satisfying $u(x,t)>0$ for all $x\in \mathbb{R}^N$, $t>0$. Moreover, mass that is conserved, $\int u(x,t)\,dx= \int u_0(x)\,dx$  for all $t>0$.

 The long term behaviour of such solutions is described by the self-similar solutions of type I with finite mass (Barenblatt solutions) constructed in \cite{VazquezBarenblattFractPME}, where it is shown that the equation admits a family of self-similar solutions, also called Barenblatt type solutions, of the form
$$
u(x,t)=t^{-N\beta_1}\phi_1(y), \quad y=x\,t^{-\beta_1}\,,
$$
and $\beta_1=1/(N(m-1)+2s)$.  Existence is proved when $m>m_c$, so that $\beta_1$ is well-defined and positive. The extra condition that is used to obtain these solutions is  $\int u(x,t)\,dx=$ constant in time. This formula produces a solution to equation (\ref{FPME}) if the profile function $\phi_1$ satisfies the following equation
\begin{equation}\label{prof1}
\laps \phi_1^m=\beta_1 \nabla \cdot(y \,\phi_1).
\end{equation}
It is proved in the above reference that the profile $\phi_1(y)$ is a smooth and positive function in $\ren$, it is a radial function, it is monotone decreasing in $r=|y|$ and has a definite decay rate as $|y|\to \infty,$ that depends on $m$ a bit, as described below.

 \begin{teor} \label{thm.Bs} For every choice of parameters $s\in(0,1)$ and $m>m_c$ and for every $M>0$, equation \eqref{FPME} admits a unique fundamental solution $u^*_M(x,t)$; it is a nonnegative and continuous weak solution for \ $t>0$ and takes the initial data $M\,\delta(x)$ as a trace in the sense of Radon measures. It has the self-similar form of type I for suitable $\alpha$ and $\beta$ that can be calculated in terms of $N$ and $s$ in a dimensional way, precisely
\begin{equation}\label{scale.expo}
\alpha=\frac{N}{N(m-1)+2s}, \qquad \beta=\frac{1}{N(m-1)+2s}\,.
\end{equation}
The profile function  $\phi_M(r)$, $r\ge 0$, is a  bounded and continuous function, it is positive everywhere, it is monotone and  it goes to zero at infinity.
\end{teor}

\noindent Moreover, the precise characterization of the profile $\phi_1$ is given by Theorem $8.1$ of \cite{VazquezBarenblattFractPME}.

\begin{teor}\label{ThJLVProfileDecay}
For every $m>m_1=N/(N+2s)$ we have the asymptotic estimate
\begin{equation}\label{decayF}
\lim_{r \rightarrow \infty} \phi_1(r)r^{N+2s}=C_1M^{\sigma},
\end{equation}
where $M=\int \phi_1(r)\,dx$, $C_1=C_1(m,N,s)>0$ and $\sigma=(m-m_1)(N+2s)\beta$. On the other hand, for $m_c<m<m_1$, there is a constant $C_{\infty}(m,N,s)$ such that
\begin{equation}\label{decayF2}
\lim_{r \rightarrow \infty}\phi_1(r) r^{2s/(1-m)}=C_{\infty}.
\end{equation}
\end{teor}
The case $m=m_1$ has a logarithmic correction. The profile $\phi_1$ has the upper bound
\begin{equation}\label{decayF3}
\phi_1(r) \leq C r^{-N-2s+\epsilon}, \quad \forall r >0
\end{equation}
for every $\epsilon>0,$ and the lower bound
\begin{equation}\label{decayF4}
\phi_1(r) \geq C r^{-N-2s}\log r, \quad \text{ for all large }r.
\end{equation}

As a consequence, the asymptotic behavior of general solutions of the Cauchy Problem for equation \eqref{FPME} is represented by such special solutions as described in Theorem 10.1 from \cite{VazquezBarenblattFractPME}.

\begin{teor}\label{ThAsympFPMEjlv}
 Let $u_0=\mu \in \mathcal{M}_+(\RN)$, let $M = \mu(\RN)$ and let $B_M$ be the self-similar Barenblatt solution with mass $M$. Then we have
\begin{equation}
\lim_{t \rightarrow \infty}|u(x,t)-B_M(x,t;M)|=0
\end{equation}
and the convergence is uniform in $\RN$.
\end{teor}

More delicate properties of general solutions to problem \eqref{FPME} have been proved recently: a priori estimates, quantitative bounds on positivity and Harnack estimates by Bonforte and Vázquez \cite{BV2012}, a priori estimates derived by Schwartz symmetrization technique by Vázquez and Volzone in \cite{VazquezVolzone2012,VazquezVolzone2013Optimal}, and numerical computations by Teso and Vázquez in \cite{Teso,TesoVaz}.

A main practical question that remains partially open is to determine if the profile $\phi_1$ can be expressed as an explicit or semi-explicit function of $r=|x|$ (and the parameters $s$ and $N$). The answer is yes in the special case $m=1$ where the solution is explicit for $s=1/2$, semi-explicit otherwise. Recently, Huang \cite{Huang2013} has shown that for every $s\in (0,1)$ there exists a certain $m_{ex}(s)>m_1$ for which the profile has an explicit expression. More precisely, $m_{ex}(s) = (N+2-2s)/(N+2s)$. For $s=1/2$ we have $m_{ex}(s)=1$, thus recovering the formula of the  linear fractional heat equation.

For $m<(N-2s)/N$, model \eqref{Model2} admits self-similar solutions of type II, as proved by V\'azquez and Volzone in \cite{VazquezVolzone2013Optimal}. Here we will prove an equivalence between \eqref{FPME} with $m>(N-2s)/N$ and \eqref{Model2} with a corresponding $\m$ interval.

\section{Model 2. First correspondence between models }\label{SectSelfSim1}

\subsection{ Preliminaries on Model 2} We call the equation of  \eqref{Model2} the \emph{Porous Medium Equation with Fractional Pressure} since it can be written as $u_t=\nabla(u^{\widetilde{m}-1}\nabla p)$ with the pressure $p=(-\Delta)^{-s}u$.

 (i) The study of the problem has been done by Caffarelli and Vázquez \cite{CaffVaz,CaffVaz2} and also with Soria \cite{CaffSorVaz} in the more natural case $\widetilde{m}=2$. Previous analysis in 1D is due to Biler et al. \cite{Biler1}. It is proved that for non-negative initial data $u_0\ge 0$, $u_0\in L^1(\ren)$, there exists a non-negative solution $u(x,t)\ge 0$.
However, uniqueness of the constructed weak solutions has not been proved but for the case $N=1$.
Moreover, the assumption of compact support on the initial data implies that the same property for all positive times, $u( \cdot,t)$ is compactly supported for all $t>0$. The existence of a self-similar solution that will be responsible for the asymptotic behaviour is obtained in \cite{Biler1} in 1D and in \cite{CaffVaz2} in all dimensions as the solution of a fractional obstacle problem. The explicit formula for this solution was given in \cite{Biler2}, and takes the form
\begin{equation}
v(x,t)=t^{-N//N+2-2s}\Phi(xt^{-1/(N+2-2s)}), \qquad \Phi(y)=(a-b|y|^2)_+^{1-s}
\end{equation}
for suitable constants $a,b>0$.

(ii) The present authors have extended these results for general $m> 1$ in \cite{StanTesoVazquezCRAS} and the preprint \cite{StanTesoVazquez}. In these recent works we prove that for non-negative initial data $u_0\ge 0$, there exists a non-negative solution $u(x,t)\ge 0$.
Different results on the positivity properties have been obtained depending on the parameter $m$ as follows:

- When $N\ge 1$, $s\in (0,1)$, $u_0\ge 0$ compactly supported and $\widetilde{m}\in [2,\infty)$ then the solution $u(x,t)$ is compact supported for all $t>0$, that is the model has finite speed of propagation.

- When $N=1$, $s\in (0,1)$, $\widetilde{m}\in (1,2)$ and $u_0\ge 0$ then the solution satisfies $u(x,t)>0$ a.e. in $\mathbb{R}$, therefore the model has infinite speed of propagation.

\medskip


\subsection{Self-similarity for Model \eqref{Model2}}

We find two main types of self-similar solutions for model \eqref{Model2} depending on the range of the parameter $\m$. The first type are functions that are positive for all times, while second type are functions that extinguish in finite time, separated by a transition type.

\medskip

\noindent $\bullet$ \textbf{Self-similarity of first type. Solutions that exist for all positive times.}
Arguing in the same way as in Model 1, or the case $\widetilde{m}=2$ of Model 2 described above, a self-similar function of the first type $v(x,t)$ is a solution to equation (\ref{Model2}) conserving mass if
\begin{equation}\label{FundSolM2type1}
v(x,t)=t^{-\alpha_2}\phi_2(y), \quad y=x\,t^{-\beta_2}
\end{equation}
with $\alpha_2=N\beta_2$  and $\beta_2=1/(N(\m-1)+2-2\s)$, and  if the profile function $\phi_2$ satisfies the equation
\begin{equation}\label{prof2}
\nabla \cdot(\phi_2^{\m-1}\,\nabla(-\Delta)^{-\s}\phi_2 )=-\beta_2 \nabla \cdot(y \,\phi_2).
\end{equation}
The existence and properties of this family of solutions have not been studied in the literature, but for the work of Huang (\cite{Huang2013}) who has shown the existence of a certain $m(s)$ for each $s \in (0,1)$ for which an explicit solution can be found.

\noindent {\bf Remark.} In the analysis below we find these solutions in the range of parameters where $\beta_2>0$, that is, for $\m>(N-2+2\s)/N$.

\medskip

\noindent $\bullet$ \textbf{Self-Similarity of second type. Extinction in finite time.}
We will also search for  solutions of the second self-similar form
\begin{equation}\label{FundSolM2type2}
v(x,t)=(T-t)^{\alphau_2} \psi_2 \left(y\right), \quad y=x(T-t)^{\betau_2}.
\end{equation}
We make again the choice $\alphau_2=N\betau_2$ even if there can be no justification in terms of mass conservation since the solutions will now extinguish in finite time (the solution to this seeming incompatibility is that the mass will be actually infinite). We use however the rule for a formal consideration: the divergence structure of the resulting profile equation will make the correspondence with Model \eqref{FPME} possible.

Let us determine the value of $\betau_2$ such that $v(x,t)$ solves the equation of \eqref{Model2}. Since
$$
v_t(x,t)=-\betau_2(T-t)^{N\betau_2-1}\nabla \cdot(y \psi_2).
$$
$$\nabla \cdot(v^{\m-1}\nabla(-\Delta)^{-\s}v)= (T-t)^{\betau_2(N\m -2\s +2)}\nabla \cdot(\psi_2^{\m-1}\nabla (-\Delta)^{-\s}\psi_2),$$
we get the compatibility condition
$$
\betau_2=1/(N(1-\m)+2\s-2)\,.
$$
The profile $\psi_2$ has to satisfy the equation
\begin{equation}\label{prof5}
\nabla \cdot(\psi_2^{\m-1}\nabla (-\Delta)^{-\s}\psi_2)=\nabla \cdot(y \psi_2).
\end{equation}

\noindent {\bf Remark.} $\betau_2=-\beta_2$, where $\beta_2$ is the self-similarity exponent of first type. We argue now in the range of parameters where $\betau_2>0$, that is $\m < (N-2+2\s)/N$.

\medskip

\noindent $\bullet$ \textbf{Self-Similarity of third type. Eternal solutions.}
There is a borderline  case  $\m= (N-2+2\s)/N $, which is not included in the previous self-similar solutions. Actually, as  $m\to(N-2+2\s)/N$ we have $1/\beta_2=1/\betau_2\to0$, and therefore self-similar solutions of the first and second type do not apply here. The possibility of self-similar representation comes from the classical porous medium equation  (see \cite{VazquezBarenblattFractPME}) where a third type of self-similar solutions of the form
\begin{equation}\label{FundSolM2type3}
v(x,t)=e^{-\gamma t}F(y), \quad y=xe^{-ct} .
\end{equation}
where $c>0$ is a free parameter (exponential self-similarity, which usually plays a transition role). We choose $\gamma=ct$ in order to have conservation of mass. It is easy to check that
$$
v_t(x,t)=-ce^{Nct}\nabla \cdot(y F ),
$$
$$\nabla \cdot(v^{\m-1}\nabla(-\Delta)^{-\s}v)= e^{-ct(-N\m+2\s-2)}\nabla \cdot(F^{\m-1}\nabla (-\Delta)^{-\s}F).$$
Then, for $m=(N-2+2\s)/N$ we get the following profile equation
\begin{equation}\label{prof6}
\nabla \cdot(v^{\m-1}\nabla(-\Delta)^{-\s}v)=-c\nabla \cdot(y F )
\end{equation}

\noindent {\bf Remark.} Solutions of this type live backwards and forward in time, they are eternal.

\subsection{Equivalence relation}

A main contribution in this paper is to show a relationship that allows to transform the families of mass-conserving self-similar solutions of models \eqref{FPME} and \eqref{Model2} into each other, if suitable parameter ranges are prescribed. Actually, the following theorem states that there exists a precise correspondence between the profiles $\phi_1$ and $\phi_2$, and the parameters $\m$ and $m$, as well as  $\s$ and $s$.

\begin{teor}\label{Thm1}Let $m>(N-2s)/N$, $s\in (0,1)$ and let $\phi_1\ge 0$ be a solution to the profile equation \eqref{prof1}. The following holds:

{\rm (i)} If $m>N/(N+2s)$ then
\begin{equation}\label{formula1}
\phi_2(x)=\left(\beta_1/\beta_2 \right)^{\frac{m}{1-m}} (\phi_1(x))^m
\end{equation}
is a solution to the profile equation \eqref{prof2} if we put $\m=(2m-1)/m$ and $\s=1-s$.

{\rm (ii)} If $m \in ((N-2s)/N, N/(N+2s))$ then
\begin{equation}\label{formula4}
\psi_2(x)=\left(\beta_1/\betau_2 \right)^{\frac{m}{1-m}} (\phi_1(x))^m
\end{equation}
is a solution to the profile equation \eqref{prof5} if we put $\m=(2m-1)/m$ and $\s=1-s$.

{\rm (iii)} If $m=N/(N+2s)$ then
\begin{equation}\label{formula4}
F(x)=\left(\beta_1/c \right)^{\frac{N}{2s}} (\phi_1(x))^{\frac{N}{N+2s}}
\end{equation}
is a solution to the profile equation \eqref{prof6} if we put $\m=(N-2+2\s)/N$ and $\s=1-s$.

\end{teor}

\noindent {\sl Comments.} The first case corresponds to exponents $\beta_1$ and $\beta_2>0$ and
produces new self-similar solutions of \eqref{Model2} are type I, i.\,e., global in time.  We see that $\beta_1>0$ if $m>(N-2s)/N$, while  $\beta_2>0$ if $\m>(N-2+2\s)/N$.
With the relation $\m=(2m-1)/m$, we have $\m>(N-2+2\s)/N$ which is equivalent to $m>N/(N+2s)$. This is another important value in the FPME, identified in \cite{VazquezBarenblattFractPME}, and we have $N/(N+2s)>(N-2s)_+/N$. Therefore, by analyzing the parameters $m$ and $\m$ for which $\beta_1>0$ and $\beta_2>0$ we have to work in the range of parameters  $ m>N/(N+2s)$.

(ii) This option  produces solutions of \eqref{Model2} that extinguish in finite time, starting with solutions of \eqref{FPME} that exist globally in time. This is a remarkable phenomenon of change of behaviour.

\medskip

\begin{proof}
(1) Let us write equation \eqref{prof1} in terms of $\phi_2$, that is, $\phi_1=\left(\beta_2 / \beta_1\right)^{\frac{1}{(1-m)}} \phi_2^{\frac{1}{m}} $, and then
\[
\laps \phi_2=\beta_2 \nabla \cdot(y \, \phi_2^{\frac{1}{m}}).
\]
Now, we pass to the parameters $\m$ and $\s$ defined by
\begin{equation}
m=\frac{1}{2-\m} \quad \text{and}\quad s=1-\s
\end{equation}
and we obtain
\[
-\Delta (-\Delta)^{-\s} \phi_2=\beta_2 \,\nabla \cdot(y \, \phi_2^{2-\m}).
\]
We can express now $\Delta$ as $\nabla \cdot \nabla$, integrate once and use the decay at infinity to transform the previous equation into the vector identity
\[
\nabla (-\Delta)^{-\s} \phi_2 =-\beta_2 \, y \, \phi_2^{2-\m}.
\]
We pass now the term $\phi_2^{\m-1}$ to the left hand side, and finally, assuming regularity on $\phi_2$  and taking divergence in both sides of the equation, we obtain
\[
\nabla \cdot(\phi_2^{\m-1}\,\nabla(-\Delta)^{-\s}\phi_2 )=-\beta_2 \,\nabla \cdot(y \,\phi_2).
\]
The regularity of $\phi_2$ follows from the already proved regularity of $\phi_1$ (\cite{VazquezBarenblattFractPME}) and the correspondence \eqref{formula1}. This is an a posteriori argument. In any case, without using the regularity, $\phi_2$ is already a weak solution of problem \eqref{Model2}.

(2) The proof is similar to the first case.
\end{proof}

\noindent {\bf Remarks.} (i) Relation between the parameters
$$
m\in [1,\infty) \longleftrightarrow \widetilde{m}\in [1,2)
$$
$$
m \in \left(\frac{N}{N+2s},1\right) \longleftrightarrow \widetilde{m} \in \left(\frac{N-2s}{N},1\right)
$$
$$
m \in \left(\frac{N-2s}{N},\frac{N}{N+2s}\right) \longleftrightarrow \widetilde{m} \in \left(\frac{N-4s}{N-2s},\frac{N-2s}{N}\right)
$$
Notice that $m=1$ implies $\m=1$, which is the \emph{Fractional Linear Heat Equation.} Since $\m_c<1$,  some singular cases of equation \eqref{Model2} are covered where $\m<1$. Thus, for $s=1/2$ and $N=2$ we get the whole range $\m\in (0,2)$.

\begin{figure}[h!]
	\begin{center}
		\includegraphics[width=\textwidth]{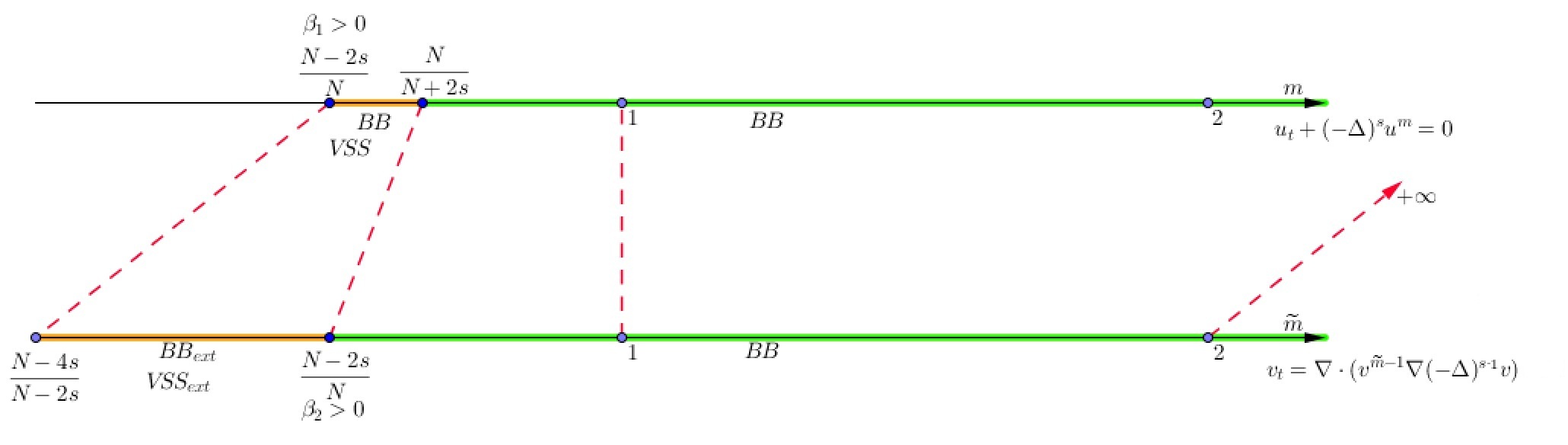}
		\caption{\footnotesize{Related profiles of  \eqref{FPME} and \eqref{Model2}. The picture is drawn for N=2 and $s=\frac{1}{2}$. The notations stand for: BB=Barenblatt solution (type I), $BB_{ext}$=Barenblatt solution with extinction in finite time (type II), VSS=Very Singular Solution, $VVS_{ext}$=Very Singular Solution with extinction in finite time.}}
        		\label{relation1}
	\end{center}
\end{figure}

\noindent(ii) Conversely we can pass from a tripe $(\s,\, \m, \, \phi_2)$ for equation \eqref{Model2} to the corresponding triple $(s, m,  \phi_1)$ for equation \eqref{FPME} through the relation
$$m=1/(2-\m), \quad s=1-\widetilde{s}, \quad \phi_1=\left(\beta_2 / \beta_1\right)^{\frac{1}{(1-m)}} \phi_2^{\frac{1}{m}} .
$$

The following corollary describes the existence ranges and asymptotic behaviour of the self-similar solutions of \eqref{Model2}. It comes as a consequence of our Theorem \ref{Thm1} and the previously known Theorem \ref{ThJLVProfileDecay}.

\begin{cor} (i) For every $\s\in (0,1)$ and $\m\in ((N-2+2\s)/N,2)$ there is a fundamental solution of equation \eqref{Model2}
given by the formula \eqref{FundSolM2type1}. The behaviour at infinity is given by
\begin{equation}
\phi_2(x)\sim C |x|^{-(N+2-2\s)/(2-\m)}.
\end{equation}
(ii) For every $\s\in (0,1)$ and $m\in ((N-4+4\s)/(N-2+2\s),(N-2+2\s)/N)$ there is a finite-time selfsimilar solution of type II  given by the formula \eqref{FundSolM2type2} with the asymptotic behaviour
\begin{equation}
\psi_2(x)\sim C |x|^{-2(1-\s)/(1-\m)}.
\end{equation}
(iii) For every $\s\in (0,1)$ and $\m=(N-2+2\s)/N)$ there is a selfsimilar eternal in time of \eqref{Model2} given by the formula \eqref{FundSolM2type3}. The behaviour at infinity is given by
\begin{equation}
F(x)\sim |x|^{-N}.
\end{equation}
\end{cor}

In all  cases the self-similar solutions have positive profiles. This is a partial confirmation that  equation \eqref{Model2} has infinite speed of propagation for all  $\m \in (\m_*,\, 2)$, $\m_*=(N-4s)/(N-2s)<1$. In the limit of this interval of infinite propagation we get the case $\m=2$, i.\,e.,  the equation studied  in \cite{CaffVaz} where finite propagation was established. Concerning general classes of solutions, we have proved infinite propagation in \cite{StanTesoVazquezCRAS} for model \eqref{Model2} for  $\m\in (1,2)$ in dimension 1. Our corollary amounts to a partial result of infinite propagation in all dimensions for a range of $\m$ that goes below 1.


\section{Model 3. Equations with finite propagation}

Now we show a relation between the profile functions for equations when finite speed of propagation is expected. This happens in the third model $\eqref{Model3}$. This model has been studied by Biler, Karch and Monneau in \cite{Biler2} for $\hat{m}=2$ describing dislocation phenomena in 1D, and for general $\hat{m}$ by Biler, Imbert and Karch in \cite{Biler1}. For non-negative initial data $u_0$ with suitable regularity properties there exists a unique weak solution $u(x,t)$ of Problem \eqref{Model3}. If $u_0\ge 0$ then also $u(x,t)\ge 0$ for $t>0$. A further characterization of the support of a general solution $u$ is not known at this point, but particular solutions are found. In \cite{Biler1}, the authors obtain a family of nonnegative explicit compactly supported
self-similar solutions the model.

Let us examine de class of mass-preserving self-similar solutions. We observe that
$$w(x,t)=t^{-N\beta_3}\phi_2(y),\quad y=x\,t^{-\beta_3}
$$
with
$$
\beta_3=1/(N(\mh-1)+2-2\sh)
$$
is a solution to  equation $\eqref{Model3}$ if the profile $\phi_3$ satisfies the equation
\begin{equation}\label{prof3}
\nabla \cdot(\phi_3\,\nabla(-\Delta)^{-\sh}\phi_3^{\mh-1} )=-\beta_3 \nabla \cdot(y\, \phi_3)
\end{equation}
and $w(x,t)$ has the property of mass conservation. Note that in the special case $\mh=2$ the equation
coincides with equation \eqref{Model2} for $\m=2$.

\noindent $\bullet$ \textbf{Case $\m=\mh=2$.} In this particular case, the equation becomes
\begin{equation}\label{Model2m2}
v_t=\nabla(v \nabla (-\Delta)^{-s}v).
\end{equation}
The existence of a family of self-similar solutions of the Barenblatt type for equation \eqref{Model2m2} with compact support in the space variable has been proved independently by Caffarelli and Vázquez in \cite{CaffVaz} and by Biler, Imbert and Karch in \cite{Biler1}. The result proves that the profile $\phi_3=\Phi$ is nonnegative, radially symmetric and compactly supported.  In the latter reference, the authors obtain an explicit formula for the self-similar solution to equation \eqref{Model2m2}  with
$$
\Phi(y)=(a-b|y|^2)_+^{1-s}\,.
$$

\noindent $\bullet$ \textbf{Case $\mh\ge 1$.} Paper \cite{Biler2} considers equation \eqref{Model3} with $\widehat m>1$ for which it presents self-similar solutions with the profile $\phi_3$ given by the explicit formula
$$
\phi_3(y)=\left(k(R^2-|y|^2)_+^{1-s}\right)^{1/(m-1)}.
$$
The regularity is H\"older continuous at the free boundary, $|x|=R\, t^{1/\beta_3}.$

Our present contribution in this instance is to show how this extension is also the result of a direct transformation of self-similarity profiles.

\begin{teor}
Let $\Phi$ be a solution to the profile equation {\rm (\ref{prof2})} with $m=2$, that is,
\begin{equation}\label{profile4}
\nabla \cdot(\Phi\, \nabla(-\Delta)^{-s}\Phi )=-\beta_2 \nabla \cdot(y \, \Phi), \qquad \beta_2= (N+2-2s)^{-1}.
\end{equation}
Let $\mh>1$. Let $\phi_3$ defined by
\begin{equation}\label{formula3}
\phi_3^{\mh-1}= (\beta_2/ \beta_3)\, \Phi.
\end{equation} Then $\phi_3$ is a solution to the profile equation
\begin{equation}\label{profile5}\nabla \cdot(\phi_3\nabla(-\Delta)^{-\sh}\phi_3^{\mh-1} )=-\beta_3 \nabla \cdot(y \,\phi_3),
\end{equation}
with $\sh=s$.
\end{teor}
\begin{proof}
As in the proof of Theorem \ref{Thm1}, from equation (\ref{profile4}) we obtain the vector identity
\begin{equation}\label{formula2}
\Phi\nabla(-\Delta)^{-s}\Phi =-\beta_2 y \Phi. \qquad y\in \RN.
\end{equation}
Let $\phi_3$ given by \eqref{formula3}, $\sh=s$ and $\mh>1$ as given in the statement of the theorem. Then from \eqref{formula2} we obtain
\[
\phi_3^{\mh-1}\,\nabla(-\Delta)^{-\sh} \,\phi_3^{\mh-1} =-\beta_3 \,y \, \phi_3^{\mh-1},
\]
and therefore
\[
\phi_3\,\nabla(-\Delta)^{-\sh}\, \phi_3^{\mh-1} =-\beta_3 \,y\,  \phi_3.
\]
We conclude that $\phi_3$ is a solution the profile equation \eqref{profile5}.
\end{proof}

This theorem proves that self-similar solutions corresponding to parameters $\mh>1$ are reduced to the computation of $\Phi$, the profile function for $\mh=2$. As a consequence of formula \eqref{formula3}, it follows that for $\phi_3=\phi_{3,\mh}$ the profile for $\mh>1$, we have
$$
\text{supp }\phi_{3} =  \text{supp }\Phi, \quad \text{for all }\mh >1.
$$
This means that the propagation of self-similar solutions does not depend on the parameter $\mh$.



\section{A more general Fractional Porous Medium Equation}

In order to see the previous transformations in a more general setting, we  will study in this section the equation
\[
z_t=\nabla(z^{\mo-1}\nabla (-\Delta)^{-\so} z^{\no-1}), \quad \tag{\ref{ModelGen}}
\]
which contains in particular the  models \eqref{Model2} and \eqref{Model3} by including two different kinds of diffusion exponents.  As far as we know, this model has not been studied before. Our goal is not to develop a complete theory of existence, uniqueness and regularity, but we want to show how the behavior of this equation is very related with the behavior of models \eqref{Model2} and \eqref{FPME}.

More specifically, we consider self-similar solutions of the form
\begin{equation}\label{selfsol4}
z(x,t)=t^{-N\beta_4}\phi_4(x t^{-\beta_4})
\end{equation}
where we have used conservation of mass as before. When considering self-similar solutions of the form \eqref{selfsol4}, the two terms of equation \eqref{ModelGen} in variable $y=xt^{-\beta_4}$ are as follows:
\[
z_t(x,t)=-\beta_4 \nabla_y\cdot(y \phi(y))t^{-N\beta_4-1},
\]
\[
\nabla_x\cdot(z^{\mo-1}\nabla_x (-\Delta)^{-\so} z^{\no-1})=\nabla_y \cdot (\phi_4^{\mo-1} \nabla_y (-\Delta)^{-\so} \phi_4^{\no-1} )t^{-\beta_4(N(\mo+\no-2)+2-2s)}.
\]

Let $\beta_4=(N(\mo+\no-3)+2-2\so)^{-1}$. The the profile $\phi_4$ is a solution to the equation
\begin{equation}\label{prof4}
\nabla \cdot(\phi_4^{\mo-1}\nabla (-\Delta)^{-\so} \phi_4^{\no-1})=-\beta_4 \nabla\cdot(y \phi_4).
\end{equation}

In the following theorem , we will prove that the self-similar solutions of this class are fully characterized when $\no >1$ by the self-similar solutions of models \eqref{FPME} and \eqref{Model2} presented in the previous sections.

\begin{teor}\label{ThmModelGen}
Let $\phi_4$ be the solution of the profile equation \eqref{prof4}
\[
\nabla \cdot(\phi_4^{\mo-1}\nabla (-\Delta)^{-\so} \phi_4^{\no-1})=-\beta_4 \nabla\cdot(y \phi_4).
\]
\begin{enumerate}
\item Let $\no>1$,$\mo<2$ and $\phi_1$ the solution to
\[\laps \phi_1^m=\beta_1 \nabla \cdot(y \phi_1),\]
for $m=\frac{\no-1}{2-\mo}$ and $s=1-\so$. Then $\phi_1=\left(\frac{\beta_1}{\beta_4}\right)^{\frac{1}{m-1}} \phi_4^{2-\mo}$ .
\item Let $\no>1$,$\mo \geq2$ and $\phi_2$ the solution to
\[\nabla \cdot(\phi_2^{\m-1}\nabla(-\Delta)^{-\s}\phi_2 )=-\beta_2 \nabla \cdot(y \phi_2),\]
for $\m=\frac{2\no+\mo-4}{\no-1}.$ and $\s=\so$. Then $\phi_2=\left(\frac{\beta_2}{\beta_4}\right)^{\frac{1}{m-1}} \phi_4^{\no-1}$.
\end{enumerate}
\end{teor}

\begin{proof}
As before, we consider the vectorial expression that can be directly deduced from $\eqref{prof4}$
\begin{equation}\label{prof4vec}
\nabla (-\Delta)^{-\so} \phi_4^{\no-1}=-\beta_4 y \phi_4^{2-\mo}.
\end{equation}
Consider also the corresponding vectorial profile equation
\[\nabla (-\Delta)^{s-1} \phi_1^m=\beta_1y \phi_1.\]
Let $\no>1$ and $\mo<2$. The profile $\phi_1=\left(\frac{\beta_1}{\beta_4}\right)^{\frac{1}{m-1}} \phi_4^{2-\mo}$ for $m=\frac{\no-1}{2-\mo}$ and $s=1-\so$ transform the last profile equation in to \eqref{prof4vec}.

In the same way, for $\no >1$ and $\mo \geq2$ the vectorial profile
\[
\nabla(-\Delta)^{-\s}\phi_2 =-\beta_2 y \phi_2^{2-\m}
\]
turns into \eqref{prof4vec} if $\phi_2=\left(\frac{\beta_2}{\beta_4}\right)^{\frac{1}{m-1}} \phi_4^{\no-1}$ for $\m=\frac{2\no+\mo-4}{\no-1}$ and $\s=\so$. We summarize the results of this theorem in Figure \ref{Figrelation4}.
\end{proof}

\begin{figure}[h!]
	\begin{center}
		\includegraphics[width=\textwidth]{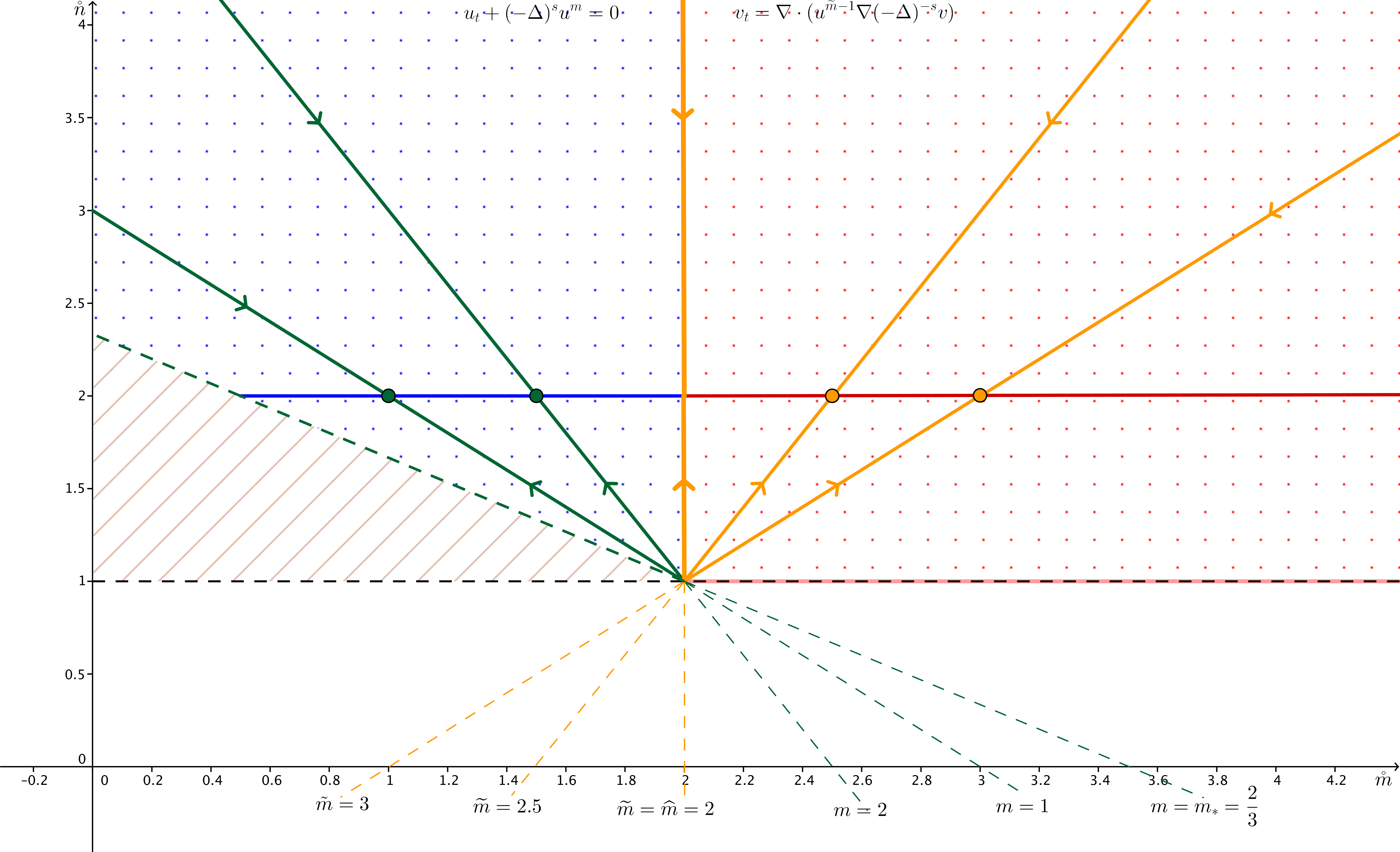}
		\caption{\footnotesize{ Related profiles of \eqref{ModelGen} with \eqref{FPME} and \eqref{Model2} for N=3 and $s=\frac{1}{2}$.
 We summarize the results on \eqref{ModelGen} as follows. The horizontal axis is the $\mo$ variable, the vertical axis is the $\no$ variable. (a) The $\mo<2$ left part describes the connection between \eqref{ModelGen} and \eqref{FPME} according to Theorem \ref{ThmModelGen}-1: each green line is defined by $\no=m(2-\mo)+1$ for particular values of $m$ that are mentioned at the lower extremity of the line. The profile function $\phi_4$ of \eqref{ModelGen}-$(\mo,\no)$ is obtained from the profile $\phi_1$ of \eqref{FPME}-($m$) for every $(\mo,\no)$ on the corresponding green line. As a consequence, $\phi_4$ for every $(\mo,\no)$ on the green line can be obtained from $\phi_4$ for $(2-1/m, 2)$, the point on the line with $\no=2$. This reduction to the case $\no=2$ is suggested by the arrows converging to the point $(2-1/m, 2)$. (b) Similarly the $\mo\ge 2$ right part describes the connection between \eqref{ModelGen} and \eqref{Model2} according to Theorem \ref{ThmModelGen}-2. Each orange line is defined by $\no(\mo-2)=\mo-4+\m$ for particular values of $\m$. The profile $\phi_4$ of \eqref{ModelGen}$-(\mo,\no)$ is obtained in this case from the $\phi_2$ of \eqref{Model2}$-(\m)$. As a consequence, $\phi_4$ of \eqref{ModelGen}$-(\mo,\no)$ is reduced to $\phi_4$ of \eqref{ModelGen}-$(\m,2)$ which is the point corresponding to $\no=2$ the the respective orange line. (c) The vertical line $\mo=2$ corresponds \eqref{Model3} studied in \cite{Biler1}. The horizontal line $\no=2$ corresponds to \eqref{Model2} studied in \cite{StanTesoVazquezCRAS}: blue for infinite propagation, while red is for finite propagation.}}
        		\label{Figrelation4}
	\end{center}
\end{figure}

\begin{note}
We should note the $N=3$ and $s=\frac{1}{2}$ plays no special role in the figure. The only determinate the shape of the line correspondent to $m=m_*=\frac{N-2s}{N}=\frac{2}{3}$.
\end{note}
\begin{note}
Notice that the relation $\phi_1=\left(\frac{\beta_1}{\beta_4}\right)^{\frac{1}{m-1}} \phi_4^{2-\mo}$ for $\no>1$ and $\mo<2$ is also true in some sense if $\no<1$ and $\mo\geq2$. It is not very clear the meaning of model \eqref{ModelGen} for $\no<1$, but this relation strongly changes the behavior of the profile equation. For example, decay at infinity does not hold anymore.

The same happens for $\phi_2=\left(\frac{\beta_2}{\beta_4}\right)^{\frac{1}{m-1}} \phi_4^{\no-1}$ if we are in the range $\no<1$ and $\mo<1$.
\end{note}


\section{Very singular solutions for \eqref{Model2}}\label{SectionVSSm2}

In this section we investigate another important class of solutions for the equation \eqref{Model2}  in a certain range of parameters $\widetilde m$ corresponding to Fractional Fast Diffusion Equations.  These solutions are called Very Singular Solutions (VSS).

\subsection{Very singular solutions of type I for \eqref{Model2}}

Solutions of this type are obtained in \cite{VazquezBarenblattFractPME}  for model M1  by the method of separation of variables. The same argument could be done here for (M2) , but instead of it, we will make use of the profile equation (\ref{prof2}) that is very familiar for us at this point. We will look for solutions of the form $\phi(y)=C|y|^{-\alpha}$ to the corresponding profile equation.
We recall that the profile equation for model (\ref{Model2}) is
\[
\nabla \cdot(\phi^{\m-1}\nabla(-\Delta)^{-\s}\phi )=-\beta_2 \nabla \cdot(y \phi). \tag{\ref{prof2}}
\]
In the sequel we will denote $y=y_i$ to clarify the vectorial notation required. The profile equation implies that
\[\phi^{\m-1}\frac{\partial}{\partial y_i}(-\Delta)^{-\s}\phi =-\beta_2y_i \phi.
\]
In Appendix 1, we review the fact that $(-\Delta)^{-\s}(|y|^{-\alpha})=k(\alpha) |y|^{-\alpha+2\s}$ where $k(\alpha)$ is an explicit constant. In this way,
\[
C^m k(\alpha) |y|^{-\alpha(\m-1)} (-\alpha+2\s) y_i|y|^{-\alpha+2\s-2}=-C\beta_2y_i|y|^{-\alpha} ,
\]
which gives the following simplified equation for the solution
\[C^{m-1}|y|^{-\alpha \m+2\s-2}=-\frac{\beta_2}{k(\alpha)(-\alpha+2\s)}|y|^{-\alpha}.\]
Therefore, $\phi$ is a solution to equation \eqref{prof2} if  the following equalities hold true
\[
-\alpha \m+2\s-2=-\alpha ,\quad C^{\m-1}=-\frac{\beta_2}{k(\alpha)(-\alpha+2\s)}.
\]
These conditions determine the exact values of the coefficient $C$ and the exponent $\alpha$:
\[
\alpha=\frac{2-2\s}{1-\m} \qquad \mbox{and} \qquad C^{1-\m}=2\overline{k}(\alpha) \frac{1-\s(2-\m)}{(1-\m)\beta_2}.
\]
where $k(\alpha)$ is calculated explicitly in appendix 1. In this way, the condition of existence of VSS for model  (\ref{Model2}) is the existence of such a constant $C$. We need to show the range of $\m$ that ensures

\begin{equation}\label{C0}
2^{1-2s}\frac{\Gamma\left(\frac{N-\alpha}{2}\right) \Gamma\left(\frac{\alpha-2\s}{2}\right)}
{\Gamma\left(\frac{\alpha}{2}\right)  \Gamma\left(\frac{N-\alpha+2\s}{2}\right)}\cdot\frac{1-\s(2-\m)}{(1-\m)\beta_2}>0.
\end{equation}
The following to properties of the $\Gamma$-functions are used
\[
\Gamma(x)>0 \mbox{ if } x>0. \qquad \Gamma(x)<0 \mbox{ if } x\in (-1,0).
\]
The first condition is $\m<1$ to make $\alpha>0$. In this way, $\Gamma(\alpha/2)>0$. We also need $\beta_2>0$ to ensure the decay at infinity of the VSS. This restriction implies the first condition on $\m$,
\begin{equation}\label{C1}
\m>\frac{N-2+2\s}{N}.
\end{equation}
We observe that for
\begin{equation}\label{C2}
\m>\frac{N+2\s}{N+2}
\end{equation}
we get $\frac{N-\alpha}{2}=-\frac{1}{2(1-\m)\beta_2}>-1$ and so on $\Gamma\left(\frac{N-\alpha}{2} \right)\frac{1}{\beta_2}<0$. It is also easy to check that $\frac{\alpha-2\s}{2}=\frac{1-\s(2-\m)}{(1-\m)}>-1$, which implies
\[\Gamma\left(\frac{\alpha-2\s}{2}\right)\frac{1-\s(2-\m)}{(1-\m)}>0 .
\]
At this point, we must have $ \Gamma\left(\frac{N-\alpha+2\s}{2}\right)<0$ to make \eqref{C1} hold. Se we need
\[
-1<\frac{N-\alpha+2\s}{2}<0\,,
\]
which gives the following two conditions on $\m$:
\begin{equation}\label{C3}
\m>\frac{N-2+4\s}{N+2\s},
\qquad
\m<\frac{N+4\s}{N+2+2\s}\,.
\end{equation}
We get the existence result by choosing the more restrictive conditions between \eqref{C1},  \eqref{C2} and  \eqref{C3}.

\begin{teor}
There exists a Very Singular Solution to model \eqref{Model2} of the form
\[v(x,t)=K t^{-\frac{1}{1-\mh}}|x|^{-\frac{2-2\s}{1-\mh}}\]
for all $\m$ in the interval $\left(\frac{N-2+4\s}{N+2\s}, \frac{N+2\s}{N+2}\right)$.
\end{teor}

\begin{proof}
We have obtained a solution $v(x,t)= t^{-N\beta_2} \phi(xt^{-\beta_2})$  for $\phi(y)=C|y|^{-\frac{2-2\s}{1-\m}}$ and exponent $\beta_2=(N(\m-1)+2-2\s)^{-1}$. The form of the solution $v$ is as follows
\[\displaystyle
v(x,t)=Kt^{\beta_2\left(\frac{2-2\s}{1-\m}-N\right) } |x|^{-\frac{2-2\s}{1-\m}},
\]
with a suitable constant $K=K(N,\s,\m)$. Since $\beta_2\left(\frac{2-2\s}{1-\m}-N\right)=-\frac{1}{1-\m}$, then $v(x,t)$ can be written in a simpler form:
\begin{equation}
v(x,t)=Kt^{-\frac{1}{1-\m}}|x|^{-\frac{2-2\s}{1-\m}}.
\end{equation}
\end{proof}
Note that these VSS are algebraically the same that the ones obtained by separation of variables in (\cite{VazquezBarenblattFractPME}) for equation (\ref{FPME}) if $\m=m$ and $\s=1-s$, except for the constant $K$.

\subsection{Very singular solutions with extinction for \eqref{Model2} }

We search for solutions of the form

$$U(x,t)=B(t)|x|^{-\alpha}$$
of equation \eqref{Model2}.
They have to satisfy the relation
$$B'(t)|x|^{-\alpha}= B(t)^{m} \mathcal{C} |x|^{-\alpha \m+2\s -2}$$
where
$$\mathcal{C}=\overline{k}(\alpha)(-\alpha+2\s)(-\alpha \m + 2\s-2+N)$$
(see the computations in the previous case).
Therefore
$$\alpha=\frac{2-2s}{1-m}$$
and
\begin{equation}\label{Bt}
B'(t)=\mathcal{C} B(t)^m.\end{equation}
We study now the sign of the constant $\mathcal{C}$:
$$\mathcal{C}=2^{1-2s}\frac{\Gamma\left(\frac{N-\alpha}{2}\right) \Gamma\left(\frac{\alpha-2\s}{2}\right)}
{\Gamma\left(\frac{\alpha}{2}\right)  \Gamma\left(\frac{N-\alpha+2\s}{2}\right)} \cdot \frac{-\alpha+2\s}{2} \cdot \frac{-1}{\beta_2(1-m)},$$
where $\beta_2^{-1}=N(m-1)+2-2\s$.

We consider now the case $\beta_2<0$ that is $\m<\frac{N-2+2\s}{N}$ and then we obtain $\alpha<N$.
Therefore all of Gamma functions above are positive, but for $\Gamma\left(\frac{\alpha-2\s}{2}\right)$ whose sign we do not estimate. We observe that
$$\Gamma\left(\frac{\alpha-2\s}{2}\right)\cdot \frac{-\alpha+2\s}{2} = - \Gamma\left(\frac{\alpha-2\s}{2}+1\right).$$
Since $\frac{\alpha-2\s}{2}+1 = (\alpha+2-2\s)/2>0$ then $\Gamma\left(\frac{\alpha-2\s}{2}+1\right)>0.$
We obtain therefore that $\mathcal{C}<0.$
Then, solving equation \eqref{Bt} with $B(T)=0$, we obtain the following formula for $B(t)$:
$$B(t)= \left(-\mathcal{C}(1-m)\right)^{1/(1-m)} (T-t)^{1/(1-m)}.$$

\begin{teor}
Let $\s \in (0,1)$, $0<\m<\frac{N-2+2\s}{N}$. Then there exists a Very Singular Solution to model \eqref{Model2} of the form
\[
v(x,t)=K (T-t)^{1/(1-m)}|x|^{-\frac{2-2\s}{1-\mh}}\quad  \text{for }t\in [0,T]
\]
and $v(x,t)=0$ for $t>T$.

\end{teor}

Note that these VSS are algebraically the same that the ones obtained by Vázquez and Volzone for \eqref{FPME} in \cite{VazquezVolzone2013Optimal} if $\m=m$ and $\s=1-s$, except for the constant $K$.


\section{Apendix 1: Inverse fractional Laplacians and Potentials }

The definition of $(-\Delta)^w $ is also done by means of Fourier transform
\begin{equation*}
\FT((-\Delta)^s f)(\xi)=(2\pi|\xi|)^{2s} \FT(f)(\xi),
\end{equation*}
and can be use even for negative values of $s$. In the range $N/2<s<0$ we have an equivalent definition in terms of a Riesz potential
\begin{equation*}
(-\Delta)^{-s} f(x)=I_s(f)=\gamma(s)^{-1}\int_{\mathbb{R}^N}\frac{f(y)}{|x-y|^{N-2s}}dy,
\end{equation*}
acting on functions of the class $\mathcal{S}$. The function $\gamma$ is defined by
\begin{equation*}
\gamma(\rho)=\pi^{N/2} \, 2^{\rho}\frac{ \Gamma(\rho/2)}{\Gamma((N-\rho)/2)}.
\end{equation*}

Note that $\gamma(\rho)\to \infty$ as $\rho \to N$, but  $\gamma(s)/(N-\rho)$ converge to the nonzero constant $\pi^{N/2}2^{N-1} \Gamma(N/2)$.

It is well known that the Fourier Transform of the function $f_\alpha(x)=|x|^{-\alpha}$ is
\[\FT(f_\alpha)(\xi)=\gamma(N-\alpha)(2\pi)^{\alpha-N}|\xi|^{\alpha-N}.\]
In this way, we can compute $(-\Delta)^{-s} f_\alpha(x)$ as follows,
\begin{eqnarray*}
\FT((-\Delta)^{-s} f_\alpha)(\xi)&=&(2\pi|\xi|)^{-2s}\FT(f_\alpha)(\xi)= \gamma(N-\alpha)(2\pi)^{\alpha-N-2s} |\xi|^{\alpha-N-2s}\\
&=&\frac{\gamma(N-\alpha)}{\gamma(N-\alpha+2s)}\gamma(N-\alpha+2s) (2\pi)^{\alpha-N-2s} |\xi|^{\alpha-N-2s}\\
&=&\frac{\gamma(N-\alpha)}{\gamma(N-\alpha+2s)} \FT(f_{\alpha-2s})(\xi),
\end{eqnarray*}
that is
\begin{equation*}
(-\Delta)^{-s} f_\alpha(x)=\overline{k}(\alpha) f_{\alpha-2s}(x), \qquad \overline{k}(\alpha)=\frac{\gamma(N-\alpha)}{\gamma(N-\alpha+2s)}.
\end{equation*}
More exactly
$$
\overline{k}(\alpha)=2^{-2s}\frac{\Gamma((N-\alpha)/2)\, \Gamma\left((\alpha-2s)/2\right)}
{\Gamma(\alpha/2)\,  \Gamma((N-\alpha+2s)/2)}.
$$

\

\section*{Comments and open problems}

The following questions appear naturally in view of the results of the paper.

\noindent $\bullet$ To decide if the asymptotic behavior of a general solution of \eqref{Model2} and \eqref{Model3} is given by a Barenblatt type solution. This fact is well known for \eqref{FPME} for general $m>(N-2s)_+/N$ (see \cite{VazquezBarenblattFractPME}) and for \eqref{Model2}, \eqref{Model3} with $\m=2$, $\mo=2$ (see \cite{CaffVaz2}).

\noindent $\bullet$ To find explicit or semi-explicit formulas for the Barenblatt profiles of models \eqref{FPME} and \eqref{Model2}.

\noindent $\bullet$ To find  explicit or semi-explicit solutions of any kind for model \eqref{Model2} with $\m>2$.

\noindent $\bullet$ Is it possible to find a transformation between general solutions of \eqref{FPME}, \eqref{Model2},  \eqref{Model3} and \eqref{ModelGen}?

\noindent $\bullet$ Develop a general theory for Model \eqref{ModelGen}.

\noindent $\bullet$ Develop numerical methods for models \eqref{Model2},  \eqref{Model3} and \eqref{ModelGen}.

\

\noindent {\large\bf Acknowledgments}

\noindent  Authors partially supported by the Spanish project MTM2011-24696. The second author is also supported by a FPU grant from MECD, Spain.
\medskip
\medskip

\

Addresses:\\
Diana Stan, {\tt diana.stan@uam.es}, \\F\'{e}lix del Teso, {\tt felix.delteso@uam.es},
\\and Juan Luis V{\'a}zquez, {\tt juanluis.vazquez@uam.es},\\
Departamento de Matem\'{a}ticas, Universidad
Aut\'{o}noma de Madrid, \\
Campus de Cantoblanco, 28049 Madrid, Spain

\end{document}